\documentclass[a4paper,11pt,english]{smfart}
\usepackage[applemac]{inputenc}
\usepackage[francais,english]{babel}
\usepackage{amsfonts}
\usepackage{amsmath} 
\usepackage{hyperref} 
\usepackage{latexsym}
\usepackage{array}
\usepackage{amssymb}
\usepackage{enumerate}
\usepackage{smfthm}
\usepackage{graphicx}
\theoremstyle{plain}

\newtheorem*{principle}{Principle}

\newcommand{\R}{  \mathbb{R}   }

\newcommand{\eps}{\varepsilon}
\newcommand{\e}{  \text{e}   }

\newcommand{\Z}{  \mathbb{Z}   }
\newcommand{\N}{  \mathbb{N}   }

\newcommand{\J}{  \mathcal{J}   }
\newcommand{\U}{  \mathcal{U}   }

\newcommand{\A}{  \mathcal{A}   }

\newcommand{\T}{  \S^{1} }
\newcommand{\Tc}{  \mathbb T}

\newcommand{\Mc}{  \mathcal{M}   }

\newcommand{\dis}{\displaystyle}

\newcommand{\om}{  \omega   }
\newcommand{\ov}{  \overline  }
\newcommand{\ga}{  \gamma  }
\renewcommand{\a}{  \alpha   }
\renewcommand{\b}{  \beta   }

\renewcommand{\phi}{  \varphi  }
\renewcommand{\L}{  \mathcal{L}   }
\newcommand{\wh}{  \widehat   }

\newcommand{\cp}[1]{  \big\{\,{ #1 }\, \big\}  }

\renewcommand{\S}{  \mathbb{S}  }
\numberwithin{equation}{section}
 
 \author{ Beno\^it Gr\'ebert}
\email{benoit.grebert@univ-nantes.fr}
 \author{ \'Eric Paturel}
\email{eric.paturel@univ-nantes.fr}
\author{ Laurent Thomann }
\email{laurent.thomann@univ-nantes.fr} 
\address{Laboratoire de Math\'ematiques J. Leray, Universit\'e de Nantes, UMR CNRS 6629\\
2, rue de la Houssini\`ere \\
44322 Nantes Cedex 03, France.}

\title[Beating effects in cubic Schr\"odinger systems]
{Beating effects in cubic Schr\"odinger systems and growth of Sobolev norms}

\begin{document}

\begin{abstract}
We consider the following coupled cubic Schr\"odinger equations 
\begin{equation*} 
\left\{
\begin{aligned}
&i\partial_t u+\partial_{x}^{2}u  = \eps^2 |v|^{2}u,\quad
(t,x)\in\R\times \S^{1},\\
&i\partial_t v+\partial_{x}^{2}v  = \eps^2 |u|^{2}v \,.
\end{aligned}
\right.
\end{equation*} 
We prove that there exists a beating effect, i.e. an energy exchange between different modes. This construction may be transported  to the linear time-dependent Schr\"odinger equation:  we build solutions such that their Sobolev norms grow logarithmically.  All these results are stated for large but finite times.
    \end{abstract}

\begin{altabstract} Nous consid\'erons le syst\`eme d'\'equations de Schr\"odinger coupl\'ees 
\begin{equation*} 
\left\{
\begin{aligned}
&i\partial_t u+\partial_{x}^{2}u  = \eps^2 |v|^{2}u,\quad
(t,x)\in\R\times \S^{1},\\
&i\partial_t v+\partial_{x}^{2}v  = \eps^2 |u|^{2}v.
\end{aligned}
\right.
\end{equation*} 
Nous montrons l'existence d'un effet de battement, c'est-\`a-dire un \'echange d'\'energie entre des modes diff\'erents.  Cette construction peut \^etre transpos\'ee  pour l'\'equation de Schr\"odinger lin\'eaire non autonome, ce qui permet de construire des solutions dont les normes de Sobolev croissent logarithmiquement (inflation de normes). Tous ces r\'esultats sont \'etablis pour des temps grands mais finis.
   \end{altabstract}

\altkeywords{Forme Normale, Equation de Schr\"odinger non lin\'eaire,
r\'esonances, \'echange d'\'energie, Equation de Schr\"odinger lin\'eaire avec potentiel d\'ependant du temps, croissance de norme.}\frontmatter
\thanks{ The three authors were supported in part by the  grant ANR-10-JCJC 0109.}

\maketitle

\section{Introduction}
\subsection{General introduction} Denote by  $\S^{1}=\R/2\pi\Z$ the circle, and let $\eps>0$ be a small parameter. In this paper we are concerned with the following cubic coupled non linear Schr\"odinger equations
\begin{equation}\label{cauchy}  
\left\{
\begin{aligned}
&i\partial_t u+\partial_{x}^{2}u  =\eps^2 |v|^{2}u,\quad
(t,x)\in\R\times \S^{1},\\
&i\partial_t v+\partial_{x}^{2}v  = \eps^2|u|^{2}v,\\
&u(0,x)=  u_{0}(x), \quad v(0,x)=  v_{0}(x)\,.
\end{aligned}
\right.
\end{equation}

We exhibit some solutions of this system which stay close to solutions of a finite dimensional {\em nonlinear} system for long times. We stress out that these solutions are not obtained by perturbations of the associated linear system.  Thanks to the nonlinearity, we may produce a beating effect, i.e. a transfer of energy between two different modes, something which is not possible in the linear case.  The solutions of the initial system are then found thanks to a resonant Birkhoff normal form and approximation arguments, and they enjoy the same beating properties as those of the reduced system. This phenomenon heavily relies on the presence of resonances. Actually, Bambusi and Gr\'ebert \cite{BG06} showed that, in the non-resonant setting (e.g. adding a typical potential in each equation of \eqref{cauchy}), the dynamics stays close to linear for long times (see the introduction of \cite{GT2}).

This new example leans on a principle that was already used in \cite{GV} and \cite{GT2}: we make it explicit in Section 2, where the conditions for applying our method on different resonant Hamiltonian PDEs are enumerated. 

The control of Sobolev norms in Hamiltonian PDEs has a long story, both in the nonlinear and the linear time-dependant setting. Concerning nonlinear equations, one of the most outstanding results is due to \cite{CKSTT} for the cubic 2-dimensional NLS, recently completed by \cite{GuKa}, where there is a construction of specific solutions which exhibit a polynomial growth of Sobolev norms for large {\em finite} times. 

In the linear setting, 
Bourgain \cite{Bou99a} proves a polynomial bound of the Sobolev norm of the solution $u$ of 
\begin{equation*}
i\partial_t u+\partial_{x}^{2}u  +V(t,x)u=0,\quad
(t,x)\in\R\times \S^{1},
\end{equation*}
where  $V(t,x)$ is a bounded (real analytic) potential. Moreover, when the potential is  quasi-periodic in time  he obtains  in \cite{Bou99b} a logarithmic bound. This last result has been enhanced by Delort \cite{Delort} and Wang \cite{Wang2008}, who gets a logarithmic bound for bounded potentials.  As a by-product of our work, we recover a result of Bourgain \cite{Bou99b}, who showed that these logarithmic bounds are optimal in the case of analytic potentials (see Section \ref{Sect4}). Note that it is possible to obtain a growth of higher order (but still logarithmic) when considering potentials in Gevrey classes (as in Fang-Zhang \cite{FangZhang}), or even a sub-polynomial growth in the case of ${\mathcal C}^\infty$ potentials. 

\subsection{Beating effect in the system \eqref{cauchy}}
Our first result concerns the dynamics of \eqref{cauchy}.
  \begin{theo}\label{thm1}
   For all $0<\gamma<1/2$, there exist $0<T_{\gamma}< C| \ln \gamma| $, a $2T_{\gamma}-$periodic function $K_{\gamma}:\R\longmapsto ]0,1[$ which satisfies $K_{\gamma}(0)=\gamma$ and $K_{\gamma}(T_{\gamma})=1-\gamma$, and there exists $\eps_{0}>0$ so that if $p,q\in \Z$ and if $0<\eps<\min(\eps_{0}, \gamma^{2})$, there exists a solution to \eqref{cauchy} satisfying for all $ |t|\leq \eps^{-3}$
   \begin{equation}\label{descr} 
   \begin{array}{lll}
 u(t,x)&=&u_{p}(t)\e^{ipx}+u_{q}(t)\e^{iqx}+ \eps^{1/2} r_u(t,x)\,,\\[5pt]
  v(t,x)&=&v_{p}(t)\e^{ipx}+v_{q}(t)\e^{iqx}+\eps^{1/2} r_v(t,x)\,,
  \end{array}
  \end{equation}
\vspace{-0,4 cm}with 
\begin{equation*} 
\begin{array}{ccccc}
 |u_{q}(t)|^{2}&=& |v_{p}(t)|^{2}&=&K_{\gamma}(\eps^{2} t)\\[5pt]
 |u_{p}(t)|^{2}&=& |v_{q}(t)|^{2}&=&1-K_{\gamma}(\eps^{2} t),
 \end{array}
 \end{equation*}
 and where $r_{u}$ and $r_{v}$ are:   
 \begin{itemize}
 \item smooth in time  and    analytic in space on $[-\eps^{-3},\eps^{-3}]\times \S^{1} $.\\
 \item for $r=r_{u},r_{v}$ the Fourier coefficients $\wh{r}_{j}(t)$ of $r(t)$ satisfy for some $\rho>0$
\begin{equation*}
\sup_{|t|\leq \eps^{-3}}|\wh{r}_{j}(t)|\leq C \e^{-\rho |j|},
\end{equation*} 
 uniformly in  $\eps>0$ and $p,q\in \Z$.
 \end{itemize}
 \end{theo}

This statement shows an exchange of energy between the modes $p$ and $q$: the mode $u_q$ and $v_p$ grow from $\gamma$ to $1-\gamma$ in time $t=\eps^{-2}T_\gamma$. Considering the larger time scales $\eps^{-3}$, we obtain a periodic phenomenon which we will call {\em beating effect}.
 
Of course the solutions satisfy  the three conservation laws: the mass, the momentum and the energy are constant quantities.  \\
$\bullet$ Conservation of the mass: $\dis \int |u|^{2}$ and $\dis \int |v|^{2}$
\begin{equation}\label{mass}
|u_{q}|^{2}+|u_{p}|^{2} = cst, \quad |v_{q}|^{2}+|v_{p}|^{2}=cst.
\end{equation}
$\bullet$ Conservation of the momentum: $\dis \text{Im}\,\int \ov{u}\partial_{x} u+\text{Im}\,\int \ov{v}\partial_{x} v$
\begin{equation}\label{moment}
q|u_{q}|^{2}+p|u_{p}|^{2}+q|v_{q}|^{2}+p|v_{p}|^{2}=cst .
\end{equation}
$\bullet$ Conservation of the energy:  $\dis \int |\partial_{x}u|^{2}+ |\partial_{x}v|^{2}+\eps^{2}\int|u|^{2}|v|^{2}$  
\begin{equation}\label{energy}
q^{2}|u_{q}|^{2}+p^{2}|u_{p}|^{2}+q^{2}|v_{q}|^{2}+p^{2}|v_{p}|^{2}=cst.
\end{equation}
On the other hand, the solutions given by Theorem \ref{thm1} satisfy for $0\leq t\leq \eps^{-3}$ and $s\geq 0$
 \begin{equation}\label{uHs}
\|u(t,\cdot)\|^{2}_{H^{s}}=(q^{2s}-p^{2s}){K_{\gamma}(\eps^{2} t)}+p^{2s}+\mathcal{O}(\eps).
\end{equation}
In particular, this norm does not remain constant in time, which is a true nonlinear effect.  However, the sum $\|u(t,\cdot)\|^{2}_{H^{s}}+\|v(t,\cdot)\|^{2}_{H^{s}}$ remains almost constant and thus \eqref{uHs} cannot be interpreted as a norm inflation. Nevertheless this effect will be used in the linear case (cf. Theorem \ref{thm3}).
  
  \begin{rema}
  For the defocusing-focusing system
   \begin{equation}\label{cauchy-} 
\left\{
\begin{aligned}
&i\partial_t u+\partial_{x}^{2}u  = \eps^{2} |v|^{2}u,\quad
(t,x)\in\R\times \S^{1},\\
&i\partial_t v+\partial_{x}^{2}v  =-\eps^{2} |u|^{2}v,\\
&u(0,x)= u_{0}(x), \quad v(0,x)=   v_{0}(x),
\end{aligned}
\right.
\end{equation} 
one can show a beating effect only for the case $q=-p$   (see also Remark \ref{rema}).
  \end{rema}

 \subsection{Growth of Sobolev norms in linear Schr\"odinger equations}
 
Theorem \ref{thm1} allows us to build real time-dependent potentials $V(t,x)$ for the following linear Schr\"odinger equation
\begin{equation}\label{lin}
i\partial_t u+\partial_{x}^{2}u  +V(t,x)u=0,\quad
(t,x)\in\R\times \S^{1}.
\end{equation}

Namely, in \eqref{cauchy} we consider the solution $v$ as a given function and we set $V(t,x) = - \eps^2 |v(t,x)|^2$. For $\alpha\geq 1$ we define the Gevrey class $\mathcal{G}_\alpha(\T)$ as the set of functions $f \in \mathcal{C}^\infty(\T)$ satisfying, for some $A>0$ and $C>0$:
\begin{equation*}
\sup_{x\in \S^{1}}|f^{(n)}(x)| \leq C A^n (n!)^{\alpha }, \quad \forall\, n \in \N\,.
\end{equation*}
In the periodic setting, an equivalent formulation is available (see \cite{yoshiko}): a function $f \in \mathcal{C}^\infty(\T)$ is in $\mathcal{G}_\alpha(\T)$ if, for some $K>0$ and $B>0$, we have for any $j \in \Z$,
\begin{equation}\label{gevrey}
|\hat{f}_j| \leq K e^{-B |j|^{1/\alpha}}\,,
\end{equation}
where $(\hat{f}_j)_{j \in \Z}$ denote the Fourier coefficients of $f$. We then define a semi-norm $ \|f\|_{\mathcal{G}_\alpha}$ as the best constant $K$ in \eqref{gevrey}  (see \cite{yoshiko} for more details).

The beating phenomenon then leads to the growth of Sobolev norms (for finite but arbitrary large times) for some solutions of this equation. Obviously, since $V$ is a real potential, the $L^2$ norm of any solution of \eqref{lin} is constant. However, we are able to prove 

 
 \begin{theo}\label{thm3}
Fix $s >0$ and $\alpha \geq 1$. There exist a sequence of real potentials $V_q(t,x)$, a sequence of initial conditions $(u_0^q)$ and a sequence of times  $T_q \longrightarrow +\infty$ as $q \longrightarrow +\infty$ such that
\begin{itemize}
\item The potentials are smooth in time, real analytic in space and uniformly bounded in Gevrey classes:
$$\forall\, q\in \Z,\;\forall\, t \in [0,T_q], \quad \|V_q(t,.)\|_{\mathcal{G}_\alpha} \leq C_{\alpha},$$
\item $\|u_0^q\|_{H^s} = 1,$
\item The corresponding solutions to the Cauchy problem $u^q(t,.)$ are real analytic in space for $t \in [0, T_q]$,
\item There exists a constant $C_{\alpha,s}$ depending only on $\alpha$ and $s$ such that
 $$ \|u^q(T_q)\|_{H^s} \geq C_{\alpha,s} (\ln T_q)^{s \alpha}\,.$$
\end{itemize}
\end{theo}
 
 This can be compared to the result obtained in the analytic case (both in time and space variables) by Wang \cite{Wang2008}, who proves that if the potential $V$ is real analytic in $(t,x)$ in a band $D=(\R + i\rho)^2$, real and bounded in $\R^2$, then, given any $s>0$, there exists $C_s$ such that
$$\|u(t)\|_{H^s} \leq C_s \big( \log(|t|+2)\big)^{\kappa s} \|u_0\|_{H^s}\,,$$
where $\kappa$ is a constant independent of $s, u_0$.

 Bourgain also showed in \cite{Bou99b} that  $\dis \|u\|_{H^{s}}\leq C(\ln t )^{Cs} \|u_{0}\|_{H^{s}}$ when $V$ is analytic and quasi-periodic in time.  This has been extended by  D. Fang- Q. Zhang \cite{FangZhang}. If $V$ is $\mathcal{C}^{\infty}$ but not supposed to be quasi-periodic in time, Bourgain \cite{Bou99a} proves the bound $\|u\|_{H^{s}}\leq C_{\eta} t^{\eta}\|u_{0}\|_{H^{s}}$, for all $\eta>0$. See also the nice generalisation by J.-M. Delort \cite{Delort}.
 \subsection{Plan of the paper}
 
 We describe in Section \ref{Sect2} the normal form method used to extract from the infinite dimensional Hamiltonian system a {\em nonlinear} finite dimensional  and {\em integrable} system which will drive our solutions. The solutions of this small system are then studied in Section \ref{Sect3}. The proof of Theorem \ref{thm1}  then relies on the control of the other terms in the initial Hamiltonian system, that is the topic of Section \ref{Sect4}. Finally  in Section \ref{Sect5} we prove Theorem \ref{thm3}.

\section{The normal form}\label{Sect2}

\subsection{Principle of the result}

In this section, we formalise the principle already used in \cite{GV} and \cite{GT2}, in order to follow for arbitrary long times solutions of an integrable model equation. The system has to be Hamiltonian: let $H$ be the Hamilton function describing its dynamics on some Hilbert phase space. We assume that $H$ is smooth in a neighbourhood of the origin, and that  its Taylor expansion is given by
\begin{equation*}
H = N + Z_p^{res} + Z_p^{nr} + R_{p+1} \,,
\end{equation*}
where 
\begin{itemize}
\item $N$ is a homogeneous polynomial of order 2, coming from the linear part of the system. It usually gathers the linear actions, i.e. the first integrals of the linearized system at the origin, which may be easily written in action-angle coordinates thanks to a Fourier transform for instance. We take the following form for clarity:
$$ N = \sum_{j \in \Z} \lambda_j I_j\,,$$
where the $\lambda_j$ are the eigenvalues of the linearization at 0 of the system. We suppose that $\lambda_j$ grows polynomially with $j$: $\lambda_j \sim |j|^r$, with $r >1$.
\item For $p$ an even integer, $Z_p^{res} + Z_p^{nr}$ is the next nonzero term in the Taylor expansion of $H$. It is a homogeneous polynomial of degree $p$.  We distinguish between {\em resonant} and {\em nonresonant} terms in the sense of Birkhoff normal forms: a monomial $M$ of degree $p$ is called resonant if it commutes with $N$, i.e. $\cp{M,N}=0$. On the contrary case, a nonresonant monomial may be removed by one step of Birkhoff normal form (see Proposition \ref{NF}): we suppose that a Birkhoff normal form is available for the system in a ball $B$ centred at the origin.
\item $R_{p+1}$ is an analytic Hamiltonian which vanishes at the origin up to order $p+1$. 
\end{itemize}
In order to observe some beating effect, we have to focus on the resonant part. Suppose that we may decompose 

$$N+Z_p^{res} = H^\square + N^{ext}+ Z_{p,1}+ Z_{p,2} + Z_{p,>2}\,,$$

where
\begin{itemize}
\item $H^\square$, defining the reduced Hamiltonian system, depends only on finitely many variables (indexed by $j \in {\mathcal A}$), called the {\it internal modes},
\item $N^{ext}$ contains all the monomials of $N$ depending on the {\it external modes}, i.e. the variables indexed by $j \not\in {\mathcal A}$,
\item $Z_{p,1}$ gathers monomials involving exactly one external mode,
\item $Z_{p,2}$ gathers monomials involving exactly two external modes,
\item $Z_{p,>2}$ gathers monomials involving at least three external modes.
\end{itemize}

We may now write the principle already used in \cite{GV} and \cite{GT2}, put in light again in this paper. This brings together the assumptions needed to exhibit beating phenomena for Hamiltonian PDEs using our method.

\begin{principle}
If the following assumptions are fulfilled:
\begin{itemize}
\item $H^\square$ defines a {\em completely integrable} Hamiltonian system, 
\item  $t \mapsto (q(t),p(t))$ is  a solution of the reduced system satisfying  that for every $t$, $(q(t),p(t))$ stays in the ball $B$,
\item $Z_{p,1} =0$, i.e. resonances cannot light on one single outer mode,
\item $\cp{N^{ext},Z_{p,2}} = 0$, i.e. $Z_{p,2}$ does not affect the external modes.
\item There exists a strictly convex combination of the $(I_j)_{j \in \mathcal{A}}$ denoted by $\mathcal{I}$ such that $\cp{\mathcal{I},H^\square}=0$.
\end{itemize}
Then there exists solutions of the system governed by $H$ which follow $(q(t),p(t))$ for long times, i.e. their projection on the reduced phase space stay close to $(q(t),p(t))$  and the difference between the solution and its projection stays small, for long times. 
\end{principle}

The beating effect is then obtained when we are able to construct  a periodic solution $t \mapsto (q(t),p(t))$ of the reduced system. Note that the 2D cubic NLS equation enters in this setting when considering "small squares" of indices, e.g. I = \{(0,0),(1,0),(0,1),(1,1)\} in the Fourier modes decomposition. We do not write the details.

\subsection{Hamiltonian formulation and Birkhoff normal form} To apply a normal form procedure it is convenient to transform the original system where the nonlinear term is small into a system where the solutions are small.
Namely by an obvious change of variable, \eqref{cauchy} is equivalent to the system
\begin{equation}\label{cauchy44}  
\left\{
\begin{aligned}
&i\partial_t u+\partial_{x}^{2}u  = |v|^{2}u,\quad
(t,x)\in\R\times \S^{1},\\
&i\partial_t v+\partial_{x}^{2}v  = |u|^{2}v,\\
&u(0,x)= \eps u_{0}(x), \quad v(0,x)= \eps v_{0}(x)\,.
\end{aligned}
\right.
\end{equation}

Denote by 
\begin{equation*}
H=\int |\partial_{x}u|^{2}+|\partial_{x}v|^{2}+\int|u|^{2}|v|^{2},
\end{equation*}
 the Hamiltonian of \eqref{cauchy44} with the symplectic structure $\text{d} u\wedge \text{d}\ov u+ \text{d} v\wedge \text{d}\ov v$. In other words, \eqref{cauchy44} is equivalent to 
\begin{equation}\label{cauchy*} 
\left\{
\begin{aligned}
&\dot u=-i\frac{\delta H}{\delta \ov u},& \quad &\dot{\ov u}=i\frac{\delta H}{\delta u},\\
&\dot v=-i \frac{\delta H}{\delta \ov v},& \quad &\dot{\ov v}=i \frac{\delta H}{\delta v}.
\end{aligned}
\right.
\end{equation} 
Let us expand $u, \bar u, v$ and $\bar v$ in Fourier modes:
\begin{eqnarray*}
u(x)=\sum_{j\in \Z}\a_j\e^{ijx},& \dis \bar  u(x)=\sum_{j\in\Z} \ov \a_j\e^{-ijx},\\
v(x)=\sum_{j\in \Z}\b_j\e^{ijx},& \dis \bar v(x)=\sum_{j\in\Z} \ov \b_j\e^{-ijx},
\end{eqnarray*}
We define
\begin{equation*}
P(\a,\b)= \int_{\T} |u(x)|^{2}|v(x)|^2 \text{d}x =  \sum_{\substack{j,\ell\in \Z^2\\ \mathcal M(j,\ell)=0}}\a_{j_1}\ov \a_{j_2}\b_{\ell_1}\ov\b_{\ell_2},
 \end{equation*}
where $\Mc(j,\ell)= j_1-j_2+\ell_1-\ell_2$ denotes the momentum of the multi-index $(\a,\b)\in \Z^{4}$. \\
In this Fourier setting the equation \eqref{cauchy*} reads as an infinite Hamiltonian system
\begin{equation}\label{hamsys}
\left\{\begin{aligned}
i\dot \a_j&=j^2\a_j+\frac{\partial P}{\partial \ov \a_j}= \frac{\partial H}{\partial \ov \a_j}, &-i\dot {\ov \a_j}&=j^2\ov \a_j+\frac{\partial P}{\partial \a_j}= \frac{\partial H}{\partial  \a_j}, &\; j\in \Z,\\
i\dot \b_j&= j^2\b_j+\frac{\partial P}{\partial \ov \b_j}=\frac{\partial H}{\partial \ov \b_j}, &-i\dot {\ov \b_j}&= j^2\ov \b_j+\frac{\partial P}{\partial \b_j}= \frac{\partial H}{\partial  \b_j}, &\; j\in \Z  .
\end{aligned}\right.
\end{equation}

For $\rho>0$, we consider   the following  phase space  
$$\mathcal{F}_\rho=\big\{\, (\a,\b)\in \big(\ell^1(\Z)\big)^{4}, s.t.\; \|(\a,\b)\|_\rho:=\sum_{j\in \Z}e^{\rho |j|}(|\a_j|+|\b_j|)<\infty \, \big\},
$$
which  we endow with the canonical symplectic structure $\dis -i\sum_j (\text{d}\a_j\wedge \text{d}\ov\a_j+\text{d}\b_j\wedge \text{d}\ov\b_j)$. 
According to this  structure, the Poisson bracket between two functions $f$ and $g$ of $(\alpha,\bar{\alpha},\beta,\bar{\beta})$ is defined by
$$\cp{f,g}={-i}\sum_{j\in\mathbb{Z}}\Big[\frac{\partial{f}}{\partial{\a_j}}\frac{\partial{g}}{\partial{\ov\a_j}}-\frac{\partial{f}}{\partial{\ov \a_j}}\frac{\partial{g}}{\partial{\a_j}}+\big(\frac{\partial{f}}{\partial{\b_j}}\frac{\partial{g}}{\partial{\ov\b_j}}-\frac{\partial{f}}{\partial{\ov\b_j}}\frac{\partial{g}}{\partial{\b_j}}\big)\Big].$$
 It is  convenient to work in the symplectic polar coordinates $\dis \big(\a_{j}=\sqrt{I_{j}}\e^{i\theta_{j}},\ov \a_{j}=\sqrt{I_{j}}\e^{-i\theta_{j}}, \b_{j}=\sqrt{J_{j}}\e^{i\phi_{j}}, \ov \b_{j}=\sqrt{J_{j}}\e^{-i\phi_{j}}\big)_{j\in \Z}$. Since we have $\text{d}\a \wedge\text{d}\ov \a=i\text{d}\theta\wedge \text{d}I$ and $\text{d}\b \wedge\text{d}\ov \b=i\text{d}\phi\wedge \text{d}J$  the system \eqref{cauchy*} is equivalent to 
\begin{equation*}
\left\{\begin{aligned}
\dot \theta_j=&-\frac{\partial H}{\partial I_j}, &\dot I_j=&\frac{\partial H}{\partial \theta_j}, &\quad j\in \Z,\\
\dot \phi_j=&-\frac{\partial H}{\partial J_j}, &\dot J_j=&\frac{\partial H}{\partial \phi_j}, &\quad j\in \Z\,.
\end{aligned}\right.
\end{equation*}

We denote by $B_\rho(r)$ the ball of radius $r$ centred at the origin in $\mathcal{F}_\rho$, and introduce the {\it resonant set}
\begin{equation*} 
{\mathcal R}=\big \{\,(j_1,j_2,\ell_1,\ell_{2})\in\mathbb{Z}^4\;s.t. \mid
 j_1-j_2+\ell_{1}-\ell_{2}=0 \;\; {\rm and } \;\; j_1^2-j_2^2+\ell^{2}_{1}-\ell^{2}_{2}=0\,\big\}\,.
\end{equation*}

\begin{prop}\label{NF}
There exists a canonical change of variable $\tau$ from $B_\rho(\eps)$ into  $B_\rho(2\eps)$ with $\eps$ small enough such that
\begin{equation}\label{ham*}
\ov{H}:=H\circ \tau = N +Z_{4}+R_{6},
\end{equation}
where
\begin{enumerate}[(i)]
\item $N$ is the term $\dis N(I) = \sum_{j\in \Z}j^2(I_{j}+J_{j})$.
\item $\dis Z_4$  is the homogeneous polynomial of degree 4     
\begin{equation*}
Z_{4}=\sum_{\mathcal{R}}\a_{j_{1}}\ov\a_{j_{2}}\b_{\ell_{1}}\ov \b_{\ell_{2}}.
\end{equation*}
In particular, $Z_{4}$ is made of resonant monomials: it satisfies $\cp{Z_{4},N}=0$.
\item  $R_{6}$ is the remainder of order 6, i.e. a Hamiltonian satisfying\\
 $\|X_{R_{6}}(z)\|_\rho \leq C \|z\|^5_\rho$   for $z=(\alpha,\bar{\alpha},\beta,\bar{\beta})\in B_\rho(\eps)$.
 \item  $\tau$ is close to the identity: there exist a constant $C_\rho$ such that $\|\tau(z)-z\|_\rho\le{C_\rho}\|z\|^2_\rho$  for all $z\in B_\rho(\eps).$
\end{enumerate}
\end{prop}
The proof is similar to the proof of Proposition 2.1 in \cite{GV}: we essentially use that if $\mathcal M(j,\ell)=0$ and $(j,\ell)\notin \mathcal R$ then $|j_1^2-j_2^2+  \ell^2_1-  \ell^2_2|\geq 1$, i.e. there is no small divisors involved.\\
For the construction of a more general Birkhoff Normal Form  see \cite{BG06}. By abuse of notation, in the proposition and in the sequel, the new variables $(\alpha',\beta')=\tau^{-1}(\alpha,\beta)$ are still  denoted by $(\alpha,\beta)$.

\subsection{Description of the resonant normal form}

In this subsection we  study the resonant part of the normal form  given by Proposition \ref{NF}. Denote by
\begin{equation*}
I=\sum_{n\in \Z}|\alpha_{n}|^{2},\quad J=\sum_{n\in \Z}|\b_{n}|^{2}, \quad S=\sum_{n\in \Z}\a_{n}\ov  \b_{n}.
\end{equation*}

\begin{prop}\label{prop.Z6}
The polynomial $Z_{4}$ reads:
\begin{equation*}
Z_{4}=IJ+|S|^{2}-\sum_{n\in \Z}|\a_{n}|^{2}|\b_{n}|^{2}.
\end{equation*}
\end{prop}

\begin{proof}
By an elementary computation, we know that $(j_1,j_2,\ell_1,\ell_{2})\in \mathcal{R}$ iff $\{j_{1},\ell_{1}\}=\{j_{2},\ell_{2}\}$ and the result follows.
\end{proof}

\section{The reduced model}\label{Sect3}
We want to describe the dynamics of the Hamiltonian system obtained by reducing   \eqref{hamsys} to the space   $$\mathcal J(p,q):=\big\{\;(\a,\b)\in\mathcal F_\rho\mid\a_{j}=\ov \a_{j}=\b_{j}=\ov \b_{j}=0 \text{ when }j\neq p,q\;\big\},$$ 
and we denote by $\wh H$ the reduced Hamiltonian, i.e.  
$$\wh{H}=H\big|_{\J(p,q)}.$$
After calculation we obtain
\begin{eqnarray*}
\wh{H}&=& p^{2}(I_{p}+J_{p})+q^{2}(I_{q}+J_{q})+(I_{p}+I_{q})(J_{p}+J_{q})+(\a_{p}\ov\a_{q}\ov\b_{p}\b_{q}+\ov\a_{p}\a_{q}\b_{p}\ov\b_{q})\\
&=&p^{2}(I_{p}+J_{p})+q^{2}(I_{q}+J_{q})+(I_{p}+I_{q})(J_{p}+J_{q})+2\big(I_{p}I_{q}J_{p}J_{q}\big)^{1/2}\cos (\psi_{0}),
\end{eqnarray*}
with $\psi_{0} =\theta_{q}-\theta_{p}+\phi_{p}-\phi_{q}$. \\
The Hamiltonian system associated to $\wh{H}$ is defined on the phase space $\Tc^4\times\R^4\ni (\theta_{p},\theta_{q},\phi_{p},\phi_{q};I_{p},I_{q},J_{p},J_{q})$ by
\begin{equation} \label{ham1}
\left\{\begin{aligned}
\dot \theta_{j}=&-\frac{\partial \wh{H}}{\partial I_{j}},\quad &\dot I_{j}=&\frac{\partial \wh{H}}{\partial \theta_{j}}, \quad &j=p,q,\\
\dot \phi_{j}=&-\frac{\partial \wh{H}}{\partial J_{j}},\quad &\dot J_{j}=&\frac{\partial \wh{H}}{\partial \phi_{j}}, \quad &j=p,q.
\end{aligned}\right.
\end{equation}

Since the Hamiltonian $\wh{H}$ only depends on one angle ($\psi_0$), the system \eqref{ham1} is completely integrable (this is also a consequence of  the invariance properties recalled in (\ref{mass})-(\ref{energy})).

\begin{lemm}\label{lem.int}
The system \eqref{ham1} is completely integrable. Moreover, the change of variables 
\begin{equation}\label{chgt0}
\left\{\begin{aligned}
&K_{1}=I_{q}+I_{p},\quad K_{2}=J_{q}+J_{p},\quad K_{3}=I_{q}+ J_{q},\quad  K_{0}=I_{q}\\
&\psi_{1}=\theta_{p}, \quad \psi_{2}=\phi_{p},\quad \psi_{3}=\phi_{q}-\phi_{p},\quad \psi_{0}=\theta_{q}-\theta_{p}+\phi_{p}-\phi_{q}
\end{aligned}\right.
\end{equation}
is symplectic: $\text{d}I\wedge \text{d}\theta+\text{d}J\wedge \text{d}\phi=\text{d} K\wedge \text{d}\psi$.
\end{lemm}

\begin{proof}
It is straightforward to check that
\begin{equation*}
K_{1}=I_{q}+I_{p},\quad K_{2}=J_{q}+J_{p} \;\;\text{and}\;\;K_{3}=I_{q}+ J_{q},
\end{equation*} 
are constants of motion. 
Furthermore we verify 
\begin{equation*}
\cp{K_{1},\wh{H}}=\cp{K_{2},\wh{H}}=\cp{K_{3},\wh{H}}=0,
\end{equation*}
as well as
\begin{equation*}
\cp{K_{1},K_{2}}=\cp{K_{2},K_{3}}=\cp{K_{3},K_{1}}=0.
\end{equation*}
Moreover the previous quantities are independent. So $\wh{H}$ admits four integrals of motions that are independent and in involution and thus $\wh{H}$ is completely  integrable.
\end{proof}

In the new coordinates, the Hamiltonian $\wh{H}$ reads 
\begin{multline}\label{Hat.gen}
\wh{H}=\wh{H}(\psi_{0},K_{0},K_{1},K_{2},K_{3})\\
\begin{aligned}
&=p^{2}(K_{1}+K_{2}-K_{3})+q^{2}K_{3}+K_{1}K_{2}+ 2\big[K(K_{3}-K)(K_{2}-K_{3}+K)(K_{1}-K)\big]^{1/2}\cos\psi_{0}.
\end{aligned}
\end{multline}

 We set $K_{1}=K_{2}=K_{3}=\eps^{2}$, and we denote by  
\begin{equation*} 
\wh{H}_{0}(\psi_0,K_0):= \wh{H}(\phi_{0},K_{0},\eps^{2},\eps^{2},\eps^{2})=\eps^{2}(p^{2}+q^{2})+\eps^{4}+2K_{0}(\eps^{2}-K_{0})\cos \psi_{0}.
\end{equation*}
The evolution of $(\psi_0,K_0)$ is given by
\begin{equation*} 
\dot \psi_{0}=-\frac{\partial \wh{H}_0}{\partial K_{0}}, \qquad
\dot K_{0}=\frac{\partial \wh{H}_0}{\partial \psi_{0}}  . 
\end{equation*}
Then, we make the change of unknown 
\begin{equation}\label{cv}
\psi_{0}(t)=\psi(\eps^{2}t)\quad \text{and}\quad K_{0}(t)=\eps^{2}K(\eps^{2}t).
\end{equation}
An elementary computation shows that the evolution of $(\psi,K)$ is given by 
\begin{equation} \label{syst*} 
\left\{\begin{aligned}
\dot \psi=&-2(1-2K)\cos \psi=&-\frac{\partial {H_{\star}}}{\partial K}\\
\dot K=&-2K(1-K)\sin \psi=&\frac{\partial {H_{\star}}}{\partial \psi},
\end{aligned}\right.
\end{equation}
where
  \begin{equation*} 
H_{\star}=H_{\star}(\psi,K) =2K(1-K)\cos\psi.
\end{equation*}

 The dynamical system \eqref{syst*} is a pendulum whose phase portrait is drawn in Figure 1 and we easily deduce that 
 
 \begin{lemm}\label{prop.mar}~
 Let $\gamma>0$ arbitrary small, then the dynamical system \eqref{syst*} admits a periodic  orbit $\Gamma_\ga:=\{(\psi_\ga(t),K_\ga(t))\mid t\in\R\}$ of period  $2T_\ga$  satisfying $(\psi_\ga(0),K_\ga(0))=(0,\ga)$ and  $(\psi_\ga(T_\ga),K_\ga(T_\ga))=(0,1-\ga)$.
\end{lemm}


\begin{figure} 
\begin{center}
\includegraphics[width=9cm]{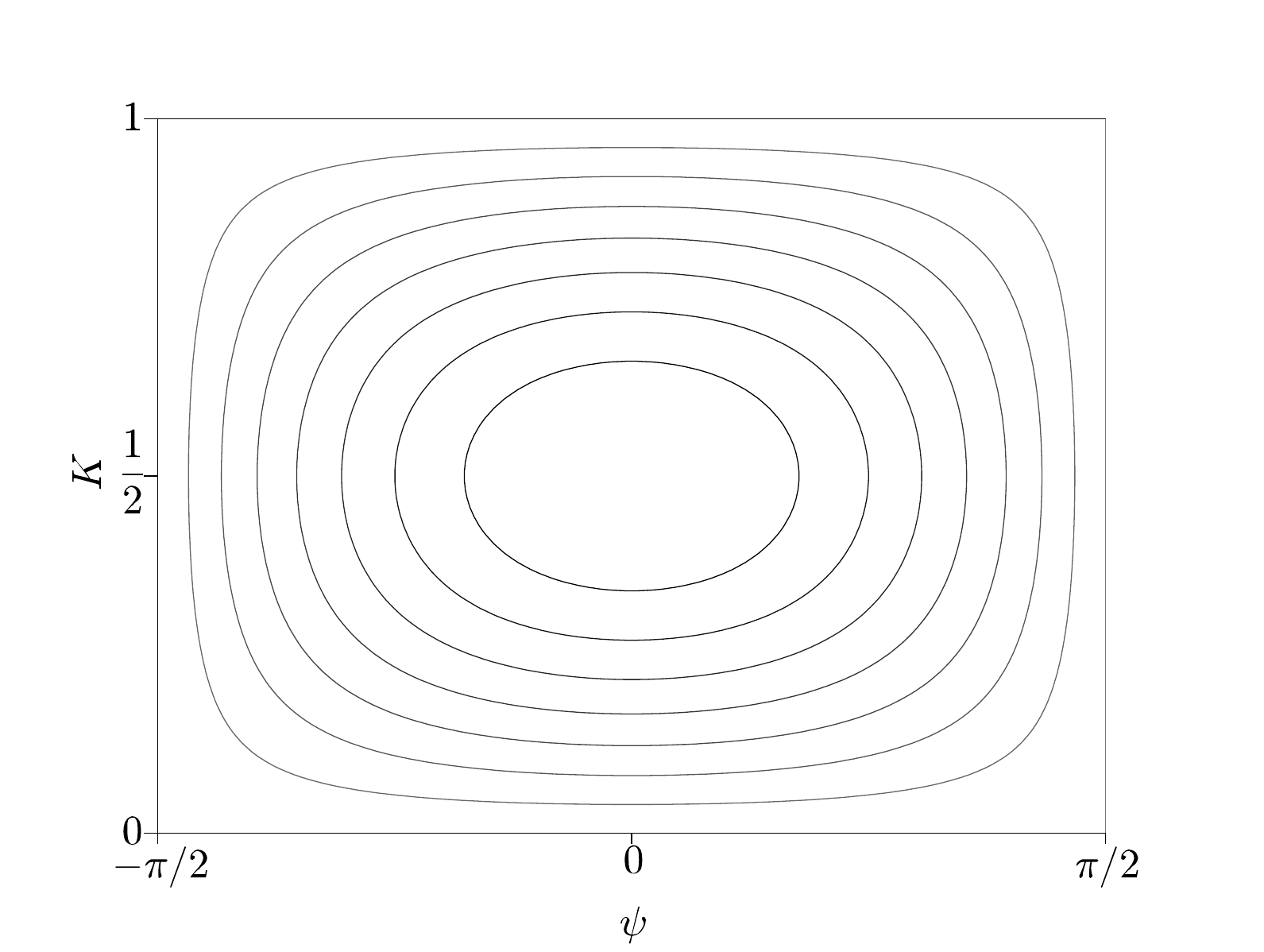}
\end{center}
\caption{Level sets of the Hamiltonian $H_\star$.}
\end{figure}
%
%

\begin{rema}\label{rema}
For the system \eqref{cauchy-}, we can perform a similar analysis, and we get the reduced Hamiltonian 
  \begin{equation*} 
H_{\star}=H_{\star}(\psi,K) = \frac2{\eps^{2}}(q^{2}-p^{2})K+2K(1-K)\cos\psi\,.
\end{equation*}
This Hamiltonian is of pendulum type for $\eps$ arbitrary small iff $q  =\pm p$,  and with  this choice only we are able to prove the same result as for \eqref{cauchy}.
\end{rema}

\section{Proof of  Theorem \ref{thm1}}\label{Sect4}
Consider the Hamiltonian $\ov{H}$ given by \eqref{ham*}, which is a function of $\big(\a_{j},\ov\a_{j},\b_{j},\ov\b_{j}\big)_{j\in \Z}$. We want to prove that, for a good choice of initial datum, the solution of the Hamiltonian system governed by $\ov H$ remains close to the solution of the reduced system governed by $\wh H$ (or $H_\star$).\\
 We make the linear change of variables given by  Lemma \ref{lem.int}. Then $\ov{H}$ induces the system
\begin{equation} \label{Bourg}
\left\{\begin{aligned}
\dot \psi_{j}=&-\frac{\partial \ov{H}}{\partial K_{j}} \\[4pt]
\dot K_{j}=&\frac{\partial \ov{H}}{\partial \psi_{j}}   
\end{aligned}\right., \quad 
 \left\{\begin{aligned}
i\dot {\a}_{k}=&\frac{\partial \ov{H}}{\partial \ov \a_{k}}, \quad & -i\dot {\ov \a}_{k}=&\frac{\partial \ov{H}}{\partial \a_{k}}\\[4pt]
i\dot{\b}_{k}=&\frac{\partial \ov{H}}{\partial \ov\b_{k}},\quad &-i\dot {\ov \b}_{k}=&\frac{\partial \ov{H}}{\partial \b_{k}}   
\end{aligned}\right.,\quad k\neq p,q. \qquad 
\end{equation}

Denote $\A=\{p,q\}$ and $\mathcal{L}= \Z \setminus \{\,p,q\,\}$. By Propositions \ref{NF} and \ref{prop.Z6} we  have
\begin{equation*}
\ov{H}=\wh{H}+R_{I}+Z_{4,2}+Z_{4,4}+R_{6},
\end{equation*}
where
\begin{equation*}
R_{I}=\sum_{j \in \mathcal{L}} (I_j +   J_j),
\end{equation*}
the polynomial $Z_{4,2}$ contains all fourth order monomials with 2 indices outside $\mathcal{A}$, and $Z_{4,4}$ contains all fourth order monomials with 4 indices outside $\mathcal{A}$. More precisely 
\begin{eqnarray*} 
Z_{4,2}&=& (I_p+I_q)\sum_{j \in \mathcal{L}} J_j + (J_p+J_q)\sum_{j \in \mathcal{L}}I_j +\\
&& +(\a_p \ov \b_p + \a_q \ov \b_q)\sum_{j \in \mathcal{L}} \ov \a_j \b_j+ (\ov \a_p \b_p + \ov \a_q \b_q) \sum_{j \in \mathcal{L}}\a_j \ov \b_j  \nonumber\\
Z_{4,4}&=& (\sum_{j \in \mathcal{L}} I_j) (\sum_{j \in \mathcal{L}} J_j) + (\sum_{j \in \mathcal{L}} \a_j \ov \b_j)(\sum_{j \in \mathcal{L}} \ov \a_j  \b_j)- \sum_{j \in \mathcal{L}} I_j  J_j\,.\nonumber
\end{eqnarray*}
Notice that $R_I$ vanishes when $I_k=0$ for all $k\notin \A$.\\
 
Observe that the $K_{j}$'s aren't constants of motion of \eqref{Bourg}. However, they are almost preserved, and this is the result of the next lemma.

 \begin{lemm}\label{lem.dyn}
 Assume that 
 \begin{equation}\label{assp}
 \a_{j}(0),\beta_{j}(0)=\mathcal{O}(\eps),\;\forall \,j\in\A\;\;\text{and}\;\;\a_{j}(0),\beta_{j}(0)=\mathcal{O}(\eps^{2}),\;\forall \,j\in\L.
 \end{equation}
Then for all $0\leq t \leq C\eps^{-3}$, 
 \begin{equation}\label{pp}
 I_{j}(t),J_{j}(t)=\mathcal{O}(\eps^{4})\;\;\text{when}\;\; j\in\L,
 \end{equation}
and
  \begin{eqnarray}
 K_{1}(t)&=&K_{1}(0)+\mathcal{O}(\eps^{6 })t\label{ip1} \\
 K_{2}(t)&=&K_{2}(0)+\mathcal{O}(\eps^{6 })t\label{ip2}\\
 K_{3}(t)&=&K_{3}(0)+\mathcal{O}(\eps^{6 })t,\label{ip3}
 \end{eqnarray}
 where the $K_{j}$'s are defined by \eqref{chgt0}.
 \end{lemm}

 \begin{proof}
 We first remark that by the preservation of the $L^2$ norm in each NLS equation, we have
 $$\sum_{j\in\Z}I_j(t) =\sum_{j\in\Z}I_j(0)\quad \text{and}\quad \sum_{j\in\Z}J_j(t) =\sum_{j\in\Z}J_j(0)\quad \text{ for all }t\in \R,$$
 and therefore by using \eqref{assp}
\begin{equation*} 
{I_{n}}(t)=\mathcal O (\eps^2),  {J_{n}}(t)=\mathcal O (\eps^2)\quad \text{for all } n\in\Z \text{ and for all }  t\in \R.
\end{equation*}
 On the other hand by Propositions \ref{NF} and \ref{prop.Z6}, we have for $n\in \Z$
\begin{equation}\label{eq0}
\dot{I_{n}}=\cp{I_{n},\ov{H}}=\cp{I_{n},Z_{4}}+\cp{I_{n},R_{6}},
\end{equation}
 and the same for $J_{n}$.

$\bullet$ To prove \eqref{pp}, we compute 
\begin{equation}\label{commut}
 \cp{I_{n},Z_{4}}=-i(\ov\a_{n}\b_{n}S-\a_{n}\ov \b_{n} \ov S),
\end{equation}
 and $\dis \cp{J_{n},Z_{4}}=i(\ov\a_{n}\b_{n}S-\a_{n}\ov \b_{n} \ov S)$. Then by \eqref{eq0}, if we denote by $L_{n}=I_{n}+J_{n}$ we get $$\dis \dot L_{n}=\cp{L_{n},R_{6}}.$$
Furthermore, for $n\neq p,q$ all the monomials appearing in $\cp{L_{n},R_{6}}$ are of order 6 and contains at least one mode in $\L$.
Therefore as soon as \eqref{pp} remains valid, we have $\dis \dot{L}_{n}(t)=\mathcal O(\eps^{2+5})$ and thus $|{L}_{n}(t)| = \mathcal O(\eps^{4})+t\ \mathcal O(\eps^{7}).$ We then conclude by a classical bootstrap argument that \eqref{pp} holds true for $t\leq C \eps^{-3}$.

$\bullet$ It remains to prove \eqref{ip1}-\eqref{ip3}. We denote by
$$Z^{e}_{4} = (I_{p}+I_{q})(J_{p}+J_{q})+(\a_{p}\ov\a_{q}\ov\b_{p}\b_{q}+\ov\a_{p}\a_{q}\b_{p}\ov\b_{q})\,,$$
the fourth order part of the model Hamiltonian. From \eqref{commut}, we get 
\begin{eqnarray*}
 \cp{I_{p},Z^{e}_{4}}&=&-i(\ov\a_{p}\b_{p}\a_{q}\ov \b_{q}-\a_{p}\ov \b_{p}\ov \a_{q}\b_{q})\\
 &=& -\cp{I_{q},Z^{e}_{4}}=- \cp{J_{p},Z^{e}_{4}}= \cp{J_{q},Z^{e}_{4}},
\end{eqnarray*}
thus $\cp{K_{1},Z^{e}_{4}} =\cp{I_{p}+I_{q},Z^{e}_{4}}=0$. Similarly, $\cp{K_{2},Z^{e}_{4}}=0$ and $\cp{K_{3},Z^{e}_{4}}=0$.
 Therefore, by using \eqref{eq0}  we deduce that for all $j\in \big\{ 1,2,3 \big\}$
 \begin{equation}\label{eqx}
 \dot{K}_{j}=\cp{K_{j},Z_{4,2}}+\cp{K_{j},R_{6}}.
 \end{equation}
 Then we use that each monomial of $Z_{4,2}$ contains at least two terms with indices $j\in \L$. Therefore, as soon as \eqref{assp} holds,  $|\cp{K_{j},Z_{4,2}}|\leq C\eps^{6}$. Furthermore $|\cp{K_{j},R_{6}}|\leq C\eps^{6}$. Therefore, by \eqref{eqx},
 $$K_{j}(t)= K_{j}(0)+t\ \mathcal{O}(\eps^{6}).$$
 \end{proof} 
 From now, we   fix the initial conditions
 \begin{equation} \label{inic}
\begin{array}{l}
K_{1}(0)=\eps^{2},\;\;K_{2}(0)=\eps^{2},\;\;K_{3}(0)=\eps^{2}, \\[4pt]
\text{and} \;\;|\alpha_{j}(0)|, \; |\bar{\alpha}_j(0)|, \; |\beta_j(0)|,\; |\bar{\beta}_{j}(0)| \leq C\eps^{2}\;\; \text{for} \;\;j\neq p,q.
\end{array} 
\end{equation}

 Let  $\ov{H}$  be given by \eqref{ham*}. Then according to the result of Lemma \ref{lem.dyn} which says that for a suitable long time we remain close to the regime of Section \ref{Sect3}, we hope that we can write $\ov{H}=\wh{H}_{0}+R$, where $R$ is an error term which remains small for times $0\leq t\leq \eps^{-3}$.
 
 We focus on the motion of $(\psi_{0},K_{0})$  and, as in the previous section, we make the change of unknown
\begin{equation}\label{renormal}
\psi_{0}(t)=\psi(\eps^{2}t)\quad \text{and}\quad K_{0}(t)=\eps^{2}K(\eps^{2}t),
\end{equation}
and we work with the scaled time variable $\tau=\eps^{2}t$. Then we can state 
\begin{prop} Consider the solution of \eqref{Bourg} with the initial conditions \eqref{inic}. Then $(\psi,K)$ defined by \eqref{renormal} satisfies for $0\leq \tau \leq \eps^{-1}$
\begin{equation} \label{syst.prop}
\left\{\begin{array}{rr}
\dot \psi=&-\frac{\partial H_{\star}}{\partial K}+\mathcal{O}(\eps^{2})\\[5pt]
\dot K=&\frac{\partial H_{\star}}{\partial \psi} +\mathcal{O}(\eps^{2})  ,
\end{array}\right. 
\end{equation}
where $H_{\star}$ is the Hamiltonian 
$$ H_{\star}=2K(1-K)\cos\psi. $$
\end{prop}

 \begin{proof}
  First recall that   $\dis \wh{H}=\wh{H}(\psi_{0},K_{0},K_{1},K_{2},K_{3})$ is the reduced Hamiltonian given by \eqref{Hat.gen}. By Propositions \ref{NF} and \ref{prop.Z6} we  have
\begin{equation}\label{413}
\ov{H}=\wh{H}+R_{I}+Z_{4,2}+Z_{4,4}+R_{6}.
\end{equation}
Thanks to  the Taylor formula there is $Q$ so that 
\begin{eqnarray}\label{415}
\wh{H}(\psi_{0},K_{0},K_{1},K_{2},K_{3})&=&\wh{H}(\psi_{0},K_{0},\eps^{2},\eps^{2},\eps^{2})+Q\nonumber \\
&=& \wh{H}_{0}+Q.
\end{eqnarray}
Thus, by \eqref{413} and \eqref{415} we have $\ov{H}=\wh{H}_{0}+R$ with 
$$R=Q+R_{I}+Z_{4,2}+Z_{4,4}+R_{6}.$$
By \eqref{Bourg}, $(\psi_{0},K_{0})$ satisfies the system
\begin{equation*} 
\left\{\begin{array}{rr}
\dot \psi_{0}(t)=&-\frac{\partial \ov{H}}{\partial K_{0}}(\psi_{0}(t),K_{0}(t),\dots) \\[5pt]
\dot {K}_{0}(t)=&\frac{\partial \ov{H}}{\partial \psi_{0}}(\psi_{0}(t),K_{0}(t),\dots)   ,
\end{array}\right. 
\end{equation*}
where the dots stand  for the dependance of the Hamiltonian on the other coordinates.  Then, after the change of variables \eqref{renormal} we obtain
\begin{equation*} 
\left\{\begin{array}{rr}
\dot \psi(\tau)=&-\frac1{\eps^{4}}\frac{\partial \ov{H}}{\partial K}(\psi(\tau),\eps^{2}K(\tau),\dots) \\[5pt]
\dot {K}(\tau)=&\frac1{\eps^{4}}\frac{\partial \ov{H}}{\partial \psi}(\psi(\tau),\eps^{2}K(\tau),\dots).
\end{array}\right. 
\end{equation*}
Now write  $\ov{H}=\wh{H}_{0}+R$ and observe that $\wh{H}_{0}(\psi,\eps^{2}K)=C_{\eps}+\eps^{4}H_{\star}(\psi,K)$. As a consequence, $(\psi,K)$ satisfies 
\begin{equation*} 
\left\{\begin{array}{rr}
\dot {\psi} =&-\frac{\partial H_{\star}}{\partial K}-\frac1{\eps^{4}}\frac{\partial R(\psi,\eps^{2}K,\dots)}{\partial K}\\[5pt]
\dot {K} =&\frac{\partial H_{\star}}{\partial \psi} +\frac1{\eps^{4}}\frac{\partial R(\psi,\eps^{2}K,\dots)}{\partial \psi}.
\end{array}\right. 
\end{equation*}
Thus it remains to estimate $\partial_{\psi} R(\psi,\eps^{2}K,\dots)$ and $\partial_{K} R(\psi,\eps^{2}K,\dots)$. Remark that $\psi$ and $K$ are dimensionless variables. Thus, if $P$ is a polynomial involving $k$ internal modes, $(\alpha_j,\bar{\alpha}_j,\beta_j,\bar{\beta}_j)_{j\in\A}$, and $\ell$ external modes, $(\alpha_j,\bar{\alpha}_j,\beta_j,\bar{\beta}_j)_{j\in\L}$, we have by using  Lemma \ref{lem.dyn}
$$\partial_{\psi} P(\psi,\eps^{2}K,\dots)=\mathcal{O}(\eps^{k+2\ell}), \quad \partial_{K} P(\psi,\eps^{2}K,\dots)=\mathcal{O}(\eps^{k+2\ell}).$$
As $R_I$ contains only monomials involving at least one external action $(I_k)_{k\notin \A}$ we get
\begin{eqnarray*}
\partial_{\psi}R_{I}(\psi,\eps^{2}K,\dots)&=\mathcal{O}(\eps^{6}),\quad \partial_{K}R_{I}(\psi,\eps^{2}K,\dots)&=\mathcal{O}(\eps^{6}),\\
\partial_{\psi}Z_{4,2}(\psi,\eps^{2}K,\dots)&=\mathcal{O}(\eps^{6}),\quad  \partial_{K}Z_{4,2}(\psi,\eps^{2}K,\dots)&=\mathcal{O}(\eps^{6}),\\
\partial_{\psi}Z_{4,4}(\psi,\eps^{2}K,\dots)&=\mathcal{O}(\eps^{6}),\quad  \partial_{K}Z_{4,4}(\psi,\eps^{2}K,\dots)&=\mathcal{O}(\eps^{6}),\\
\quad \partial_{\psi}R_{6}(\psi,\eps^{2}K,\dots)&=\mathcal{O}(\eps^{6}),\quad \partial_{K}R_{6}(\psi,\eps^{2}K,\dots)&=\mathcal{O}(\eps^{6}).
\end{eqnarray*}
On the other hand, by construction  $Q$ reads $P_1\Delta K_1+ P_2 \Delta K_2+P_{3}\Delta K_{3}$ where $P_1$, $P_2$ and $P_{3}$ are polynomials of order 1 in $K_0$, $K_1$, $K_2$, $K_{3}$ and $\eps^2$ while $\Delta K_j$ denotes the variation of $K_j$: $\Delta K_j= K_j-K_j(0)$. Using again Lemma \ref{lem.dyn}, we check that for $0\leq \tau\leq \eps^{-1}$
\begin{equation*}
 \quad \partial_{\psi}Q=\mathcal{O}(\eps^{2}),\quad \partial_{K}Q=\mathcal{O}(\eps^{2}),
\end{equation*}
hence the result.
\end{proof}

We now consider the solution $(\psi_\gamma,K_\gamma)$ of \eqref{syst*}, described in Lemma \ref{prop.mar}, which is issued from the initial condition $(\psi_\gamma,K_\gamma)(0) = (0,\gamma)$ for some $\gamma$ such that $\eps^{1/2} \leq \gamma \ll 1$ and we compare it with  the solution $(\psi,K)$ of \eqref{syst.prop} issued from the same initial datum:

\begin{lemm}\label{lem.compa}
For all $0\leq \tau \leq \eps^{-1}$ we have
\begin{equation}\label{11} 
 ({\psi},{K})(\tau)=(\psi_\ga,K_\ga)(\tau)+ \mathcal{O}(\eps^2)\tau \,.
\end{equation}
\end{lemm}

\begin{proof} Consider the system \eqref{syst*} and the open domain $\U= (-\pi,\pi)\times(0,1)$. By the Arnold Theorem (cf. \cite[p.113]{Arn}, see also   \cite[Lemma 4.3]{GT2}), this Hamiltonian system admits  action-angle coordinates $(L,\alpha) = \Phi(\psi,K)$ defined on $\U\setminus\{(0,0)\}$ by a $\mathcal{C}^1$ symplectic map $\Phi$ satisfying that uniformly  on any compact $\tilde\U\subset \U\setminus\{(0,0)\}$:
$$ \|\text{d} \Phi\| \leq C, \quad \|\text{d} \Phi^{-1} \| \leq C\,.$$
Then we obtain that for $0\leq \tau\leq \eps^{-1}$
\begin{eqnarray*} 
\frac{\text{d}}{\text{d}\tau}(L,\alpha)&=&\frac{\text{d}}{\text{d}\tau}\Phi(\psi,K)= \text{d}\Phi(\psi,K). (\dot \psi,\dot K)\\&=&\text{d}\Phi (\psi,K). (\frac{\partial H_{\star}}{\partial \psi},-\frac{\partial H_{\star}}{\partial K})+\mathcal{O}(\eps^{2})\\
&=&(\frac{\partial H_{\star}}{\partial \alpha},-\frac{\partial H_{\star}}{\partial L})+\mathcal{O}(\eps^{2})\\
&=&(0,-\frac{\partial H_{\star}}{\partial L})+\mathcal{O}(\eps^{2}).
\end{eqnarray*}
Therefore there exists $L_{\star}\in \R$ so that $L(\tau)=L_{\star}+\mathcal{O}(\eps^{2})\tau$ and if we define  
 $\omega_{\star}=-\frac{\partial H_{\star}}{\partial L}(L_{\star}) $, we obtain  $\dis \alpha(\tau)=\omega_{\star}\tau+ \mathcal{O}(\eps^{2})\tau$. Notice that by construction $\Phi(\psi_\ga(\tau),K_\ga(\tau))= (L_\star,\om_\star \tau)$ for all $\tau\in \R$. Next, as $\text{d}\Phi^{-1}$ is bounded, we get 
\begin{eqnarray*} 
 ({\psi},{K})(\tau)=\Phi^{-1}\big(L(\tau),\alpha(\tau)\big) &=&\Phi^{-1}\big(L_{\star},\omega_{\star}\tau\big)+ \mathcal{O}(\eps^{2})\tau\nonumber \\
 &=& (\psi_{\ga},K_{\ga})(\tau)+ \mathcal{O}(\eps^{2})\tau.
\end{eqnarray*}
With the choice $\eps^{1/2}\leq \gamma$, the remainder term in \eqref{11} is so that $|\mathcal{O}(\eps^{2})\tau|\ll \gamma\leq K_{\gamma}$ for $\tau\leq \eps^{-1}$. 
\end{proof}

\begin{proof}[Proof of Theorem  \ref{thm1} ]
As a consequence of Lemma \ref{lem.compa}, the solution of \eqref{Bourg}, with initial datum \eqref{inic} and $(\psi_0,K_0)(0) = (0,\eps^2\gamma)$, satisfies for $0\leq t\leq \eps^{-3}$
\begin{eqnarray*}
K_{0}(t)&=&\eps^{2}K_{\ga}(\eps^{2}t)+ \mathcal{O}(\eps^{6}t)  \\
\psi_{0}(t)&=&\psi_{\ga}(\eps^{2}t)+ \mathcal{O}(\eps^{4}t),
\end{eqnarray*}
and with the condition $\eps^{1/2}\leq \gamma$ we obtain \eqref{descr}.\\
We now compute the period $2T_\gamma$. From the expression of the Hamiltonian $H_\star$, we infer
$$ \dot K = 2 K (1-K) \sqrt{1 - \frac{h^2}{\big( 2 K (1-K) \big )^2}}\,,$$
where $h = 2 \gamma(1-\gamma)>0$.  Thanks to the symmetries of  $H_\star$, ${T_\gamma}/{2}$ is the travel time for the solution between $(0,\gamma)$ and $(-\cos^{-1}(2h), {1}/{2})$. In the interval $[0, {T_\gamma}/{2}]$, the function $K$ is strictly increasing, hence invertible, so we can write the time $t$ as a function of $K$, and this implies that
\begin{equation}\label{integrale}
T_\gamma  = 2 \int_\gamma^{\frac{1}{2}} \frac{dK}{ \sqrt{(2K(1-K))^2-h^2}}\,. 
\end{equation}
Next, we estimate $T_{\gamma}$.    It is easy to check that there exists $C>0$ such that for all $\gamma < K < \frac{1}{2}$, we have
$$ K^2(1-K)^2 - \gamma^2(1-\gamma)^2 \geq C(K^2-\gamma^2) \,.$$
Hence by integration in \eqref{integrale} we deduce that there exists $C>0$ so that
\begin{equation}\label{log}
T_\gamma \leq - C \ln \gamma\,.
\end{equation}

\end{proof}
\section{Proof of Theorem \ref{thm3}}\label{Sect5}

Fix $\alpha \geq 1$ and consider the system \eqref{cauchy}. We fix the set of internal modes: $p=0$, $q \gg 1$  and 
$$\eps:=\e^{-\frac{1}{2}q^{1/\alpha}}.$$ 
Then we consider the first equation in \eqref{cauchy} as a linear time-dependent Schr\"odinger equation 
\begin{equation*}
i\partial_t u+\partial_{x}^{2}u  +V_q(t,x)u=0,\quad
(t,x)\in\R\times \S^{1},
\end{equation*}
with potential $V_q=-\e^{-q^{1/\alpha}} |v|^{2}$.\\ 
The regularity properties of the solutions given in Theorem \ref{thm1} imply that, on $[0,\eps^{-3}]\times \S^{1} $ the function   $V_q$ is smooth in time  and   real analytic in space. 
In order to construct the sequence of initial conditions announced in Theorem \ref{thm3}, we have to ensure uniform bounds w.r.t. the integer $q \in \N$ in  the Gevrey class $\mathcal{G}_\alpha(\T)$, given in \eqref{gevrey}.  We have 
$$ V_q(t,x) = \e^{-q^{1/\alpha}}\big| v_0(t) + v_q(t) \e^{iqx} + \e^{-\frac{1}{4}q^{1/\alpha}}r_v(t,x) \big|^2\,,$$
where $r_v(t,.)$ is an analytic function, whose norm is uniformly bounded with respect to $ |t| \leq \eps^{-3}$ and $q \in \N$. We then compute the Fourier coefficients $\hat{V}_j$ of $V_q(t,.)$.

The dominant coefficients are labelled by the indices $0$, $q$ and $-q$: for them, we have 
$$ |\hat{V}_0| \leq \e^{-q^{1/\alpha}} ( |v_0|^2+|v_q|^2 + \e^{-\frac{1}{4}q^{1/\alpha}}R)\,,$$
$$ |\hat{V}_q| \leq \e^{-q^{1/\alpha}} (|v_0| |v_q| + \e^{-\frac{1}{4}q^{1/\alpha}}R')\,,$$
$$ |\hat{V}_{-q}| \leq \e^{-q^{1/\alpha}} (|v_0 ||v_q| + \e^{-\frac{1}{4}q^{1/\alpha}}R'')\,,$$
where $R, R', R''$ are uniformly bounded w.r.t. $q\in \N$ and $0\leq t\leq \eps^{-3}$. Since $v_0$ and $v_q$ stay (in modulus) between $0$ and $1$, estimate \eqref{gevrey} is obtained for these indices. 

The coefficients $\hat{V}_j$ for $j \neq -q,0,q$ decay much faster in general: using the analyticity of $r_v(t,.)$, and the fact that for every $q\in \N$ and $\ell\in \Z$ 
$$ \frac{5}{4}|q|^{1/\alpha} + \rho |\ell-q| \geq c |\ell|^{1/\alpha}\,,$$
 we obtain
 $$| \hat{V}_j| \leq C \e^{-c|j|^{1/\alpha}}\,.$$
Once again, this estimate is uniform w.r.t. $q \in \N$ and $0\leq t\leq \eps^{-3}$.

Choose initial conditions so that $(\psi(0),K(0))=(0,\gamma)$ for some $0<\gamma\ll 1$.  In order to apply the result of Theorem \ref{thm1} we must have     $T_\gamma \leq \eps^{-1}$. 
 Therefore  we impose 
$$ T_q := \frac{T_\gamma}{\varepsilon^2} < \frac{1}{\varepsilon^3}\,,$$
which leads to $\e^{-\frac{1}{2}q^{1/\alpha}} < \frac{C}{|\ln \gamma|}$ since $T_{\gamma}\leq C |\ln \gamma|$. We fix the   $H^s$-norm of the initial condition with the choice $\gamma = q^{-2s}$ (observe that for $q\gg1$, $\eps<\gamma^{2}$, so that we are in the conditions of application of Theorem \ref{thm1}), then the previous constraint becomes $\e^{-\frac{1}{2}q^{1/\alpha}} <  \frac{C}{2s \ln q}$, hence is satisfied for $q$ large enough. From \eqref{log}, we get
$$ T_q \leq C |\ln \gamma |\eps^{-2} = 2s C\e^{q^{1/\alpha}} \ln q\,.$$
Now, the growth rate of $\|u\|_{H^{s}}$ between $t=0$ and $t=T_q$ is bounded from below by $C'\gamma^{-1/2}$, where $C'$ is independent of $s$. So we have, by \eqref{uHs},
\begin{equation*}
\frac{\|u(T_q)\|_{H^s}}{\|u(0)\|_{H^s}} \geq C'\gamma^{-1/2} = C' q^s \geq \frac{C'}{(1 + 2\alpha s)^{\alpha s}}  (\ln T_q)^{\alpha s}\,.
\end{equation*}
Note that the constant $C_{\alpha,s}:= \frac{C'}{(1 + 2\alpha s)^{\alpha s}}$ goes to $0$ as $s$ or $\alpha$ goes to infinity.
\qed

\begin{rema}
If we choose a different $\eps$, as for example $\eps = \exp{(- (\ln q)^{1+\kappa})}$, with $\kappa >0$, we have that for all $s>0$

$$\forall\, q \in \N,\quad \forall \,t \in [0,T_q], \quad \|V_q(t,.)\|_{H^s} \leq C_{s,\alpha}\,,$$
and we obtain the growth 
\begin{equation*}
\frac{\|u(T_q)\|_{H^s}}{\|u(0)\|_{H^s}} \geq C \exp{\big(s(\ln T_q)^{1/(1+\kappa)}\big)}\,,
\end{equation*}
that is, a sub-polynomial growth.
\end{rema} 

\end{document}